\documentclass{amsart}
\usepackage{amsmath, amsfonts, amssymb, amsthm, amscd}

\newtheorem{theorem}{Theorem}[section]
\newtheorem{lemma}[theorem]{Lemma}
\newtheorem{corollary}[theorem]{Corollary}

\newtheorem{proposition}[theorem]{Proposition}

\theoremstyle{definition}
\newtheorem{definition}[theorem]{{Definition}}
\newtheorem{example}[theorem]{Example}

\newtheorem*{chunk*}{}

\numberwithin{equation}{section}

\theoremstyle{remark}
\newtheorem{remark}[theorem]{Remark}

\input xy
\xyoption{all}

\def\ann{\operatorname{ann}}
\def\Ass{\operatorname{Ass}}

\def\diam{\operatorname{diam}}
\def\p{\mathfrak{p}}

\begin{document}
\title[equivalence classes of zero divisors]
{A zero divisor graph determined by equivalence classes of zero divisors}

\author{Sandra Spiroff}
\address{Department of Mathematics, University of Mississippi, Oxford, MS  38677}
\email{spiroff@olemiss.edu}

\author{Cameron Wickham}
\address{Department of Mathematics, Missouri State University, Springfield, MO  65897}
\email{cwickham@missouristate.edu}

\begin{abstract} We study the zero divisor graph determined by equivalence classes of zero divisors of a commutative Noetherian ring $R$.  We demonstrate how to recover information about $R$ from this structure.  In particular, we determine how to identify associated primes from the graph.
\end{abstract}
\maketitle

\section*{Introduction} Beck first introduced the notion of a zero divisor graph of a ring $R$ in 1988 \cite{B} from the point of view of colorings.  Since then, others have studied and modified these graphs, whose vertices are the zero divisors of $R$, and found various properties to hold.  Inspired by ideas from S. Mulay in \cite[\S 3]{M}, we introduce the {\it graph of equivalence classes of zero-divisors of a ring $R$}, which is constructed from classes of zero divisors determined by annihilator ideals, rather than individual zero divisors themselves.  It will be denoted by $\Gamma_E(R$).

This graph has some advantages over the earlier zero divisor graphs $\Gamma(R)$ in \cite{B}, \cite{AL}, \cite{AM}, \cite{AMY}, or subsequent zero divisor graphs determined by an ideal of $R$ in \cite{MPY}, \cite{R}.  In many cases $\Gamma_E(R)$ is finite when $\Gamma(R)$ is infinite.
For example, if $S = \mathbb Z[X,Y]/(X^3, XY)$, then $\Gamma(S$) is an infinite graph, while $\Gamma_E(S$) has only four vertices.  To be specific, although $X^2, 2X^2, 3X^2, \dots$, are distinct zero divisors, they all have the same annihilator; they
are represented by a single vertex in $\Gamma_E(S$).  In addition, there are no complete $\Gamma_E(R$) graphs with three or more vertices since the graph would collapse to a single point.  These are two ways in which $\Gamma_E(R$) represents a more succinct description of the \lq\lq zero divisor activity" in $R$.

Another important aspect of graphs of equivalence classes of zero divisors is the connection to associated primes of the ring.  For example, in the ring $S$ above, the annihilator of $X^2$ is an associated prime.  In general, all of the associated primes of a ring $R$ correspond to distinct vertices in $\Gamma_E(R)$.  Moreover, every vertex in a graph either corresponds to an associated prime or is connected to one.
The study of the structure of associated primes in $\Gamma_E(R$) is one of our main motivations.

In section one, we compare and contrast $\Gamma_E(R$) with the more familiar $\Gamma(R$) defined by D.~Anderson and P.~Livingston \cite{AL}.  In section two, we consider infinite graphs and fan graphs and answer the question of whether or not the Noetherian condition on $R$ is enough to force $\Gamma_E(R)$ to be finite.  Section three is devoted to the relation between the associated primes of $R$ and the vertices of $\Gamma_E(R)$.  In particular, we demonstrate how to identify some elements of $\Ass(R)$.

Throughout, $R$ will denote a commutative Noetherian ring with unity, and all graphs are simple graphs in the sense that there are no loops or double edges.

\section{Definitions and Basic Results}  Let $Z^*(R)$ denote the zero divisors of $R$ and
$Z(R)=Z^*(R) \cup \{0\}$.  For $x,y\in R$, we say that $x \sim y$
if and only if $\ann(x) = \ann(y)$. As noted in \cite{M}, $\sim$
is an equivalence relation.  Furthermore, if $x_1\sim x_2$ and
$x_1y=0$, then $y\in\ann(x_1)=\ann(x_2)$ and hence $x_2y=0$.  It
follows that multiplication is well-defined on the equivalence classes of $\sim$;
that is, if $[x]$ denotes the class of $x$, then the product $[x]
\cdot [y] = [xy]$ makes sense.  Note that $[0]=\{0\}$ and
$[1]=R-Z(R)$; the other equivalence classes form a partition of
$Z^*(R)$.

\begin{definition}  The {\it graph of equivalence classes of zero-divisors of a ring $R$}, denoted $\Gamma_E(R$), is the graph associated to $R$ whose vertices are the classes of elements in $Z^*(R)$, and with each pair of distinct classes $[x], [y]$ joined by an edge if and only if $[x] \cdot [y]=0$.
\end{definition}

Recall that a prime ideal $\p$ of $R$ is an \textit{associated
prime} if $\p = \ann(y)$ for some $y \in R$.  The set of
associated primes is denoted $\Ass(R)$; it is well known that for a Noetherian ring $R$, $\Ass(R)$ is finite, and any maximal element of the family of ideals $\mathfrak F = \{\ann(x) |
0 \neq x \in R \}$ is an associated prime.  There is a natural injective map from $\Ass(R)$ to the
vertex set of $\Gamma_E(R)$ given by $\p \mapsto [y]$ where $\p = \ann(y)$.
As a result, we will slightly abuse terminology and refer to
$[y]$ as an associated prime.

\begin{lemma}  \label{assedge} Any two distinct elements of $\Ass(R)$ are connected by an
edge. Furthermore, every vertex $[v]$ of $\Gamma_E(R)$ is either an
associated prime or adjacent to an associated prime maximal in $\mathfrak F$.
\end{lemma}

\begin{proof}  The proof of the first statement is essentially the same as in \cite[Lemma 2.1]{AMY}.  If $\frak p = \ann(x)$ and $\frak q = \ann(y)$ are primes, then one can assume that there is an element $r \in \frak p \backslash \frak q$.  Since $rx=0 \in \frak q$, $x \in \ann(y)$, and hence $[x] \cdot [y] = 0$.  Next, suppose $[v] \in \Gamma_E(R)$ is not an associated prime.  Since $v$ is
a zero divisor, $v \in \ann(z)$ for some $z$ maximal in $\frak F$; thus, there is an edge between $[v]$ and [z].
\end{proof}

\begin{example}  \label{house}  The associated primes of $R = \mathbb Z_4 \times \mathbb Z_4$ are $[(2,0)]$ and $[(0,2)]$.

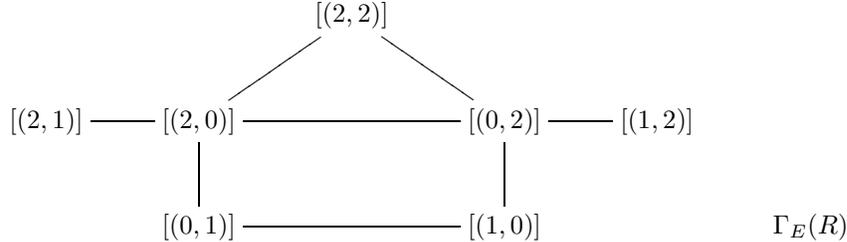
\xymatrix{
& & & [(2,2)] \ar@{-}[dl] \ar@{-}[dr]  \\
& [(2, 1)] \ar@{-}[r] & [(2,0)] \ar@{-}[rr] \ar@{-}[d] & &  [(0,2)] \ar@{-}[d] \ar@{-}[r] & [(1, 2)] \\
 & & [(0,1)] \ar@{-}[rr] & &  [(1, 0)] & & \Gamma_E(R) }
\end{example}

The {\it degree} of a vertex $v$ in a graph,
denoted $\deg v$, is the number of edges incident to $v$.  When $\deg v = 1$, we call $v$ a {\it leaf}.  A {\it path of length $n$} between two vertices $u$ and $w$ is a sequence
of distinct vertices $v_i$ of the form $u=v_0-v_1- \ldots- v_n=w$
such that $v_{i-1}-v_i$ is an edge for each $i=1,\ldots, n$. The
{\it distance} between a pair of vertices is the length of the shortest
path between them; if no path exists the distance is infinite. The
{\it diameter} of a graph is the greatest distance between any two
distinct vertices.

\begin{proposition}  \label{connected} The graph $\Gamma_E(R$) is connected and $\diam \Gamma_E(R) \leq 3$.
\end{proposition}

\begin{proof}  Let $[x]$ and $[y]$ be two non-adjacent vertices.  At worst, neither $[x]$ nor $[y]$ is an associated prime.  By Lemma \ref{assedge}, there are associated primes $[v_1]$ and $[v_2]$ adjacent to $[x]$ and $[y]$, respectively, providing a path $[x]-[v_1]-[v_2]-[y]$,
where $[v_1]$ may or may not equal $[v_2]$.  In any case, $\Gamma_E(R$) is connected and
$\diam \Gamma_E(R) \leq 3$.
\end{proof}

Connectivity and a restricted diameter are two similarities between $\Gamma_E(R)$ and $\Gamma(R)$, \cite{AL}.  One important difference between $\Gamma_E(R)$ and $\Gamma(R)$ lies in the study of complete graphs.  The idea in the next proposition is that complete zero divisor graphs determined by equivalence classes, with the exception of graphs consisting of a single edge (e.g.,  $\Gamma_E(\mathbb Z_{pq})$; $\Gamma_E(\mathbb Z_{p^3})$), all collapse to a single point.  Each vertex in $\Gamma_E(R)$ is  representative of a distinct class of zero divisor activity in $R$.

\begin{proposition}  \label{incomplete} Let $R$ be such that $\Gamma_E(R$) has at least three vertices.  Then $\Gamma_E(R)$ is not complete.
\end{proposition}

\begin{proof}  Suppose $\Gamma_E(R$) is a complete graph and $[x], [y],$ and $[z]$ are three distinct vertices.  We may assume that $\ann(z)$ is a maximal element of $\frak F$ and hence is not contained in either $\ann(x)$ or $\ann(y)$.  Then there is an element $w \in \ann(z)$ such that $w \notin \ann(x) \cup \ann(y)$.   This immediately contradicts the assumption that the graph is complete.
\end{proof}

\begin{corollary}  \label{three} There is only one graph $\Gamma_E(R$) with exactly three vertices that can be realized as a graph of equivalence classes of zero divisors for some ring $R$.
\end{corollary}

\begin{proof}  Since every graph is connected, by Proposition \ref{incomplete}, we need only note that the graph of $R = \mathbb Z_4[X,Y]/(X^2, XY, 2X)$, for example, is $[2]$ --- $[x]$ --- $[y]$, where $x$ and $y$ denote the images of $X$ and $Y$, respectively.
\end{proof}

The simplification, or collapse, of complete $\Gamma_E(R$) graphs applies also to complete $r$-partite diagrams, with the exclusion of fan graphs $K_{n,1}$, like the one above.  (The latter are considered in the next section.)  A graph is {\it complete bipartite} if there is a partition of the vertices into two subsets $\{u_i\}$ and $\{v_j\}$ such that $u_iv_j = 0$ for all pairs $i,j$, but no two elements of the same subset annihilate one another.  More generally, a graph is {\it complete $r$-partite} if the vertices can be partitioned into $r$ distinct subsets such that each element of a subset is connected to every element not in the same subset, but no two elements of the same subset are connected.

\begin{proposition}  \label{bipartite} Let $R$ be a ring such that $\Gamma_E(R)$ is complete $r$-partite.  Then $r=2$ and $\Gamma_E(R) = K_{n,1}$ for some $n\geq 1$.
\end{proposition}

\begin{proof}  First suppose $\Gamma_E(R) = K_{n_1, n_2, \dots, n_r}$ for some $r\geq 3$.  Proposition \ref{incomplete} implies that not all $n_i = 1$; without loss of generality assume $n_1>1$ and let $[u_1], [u_2]$ be two non-adjacent vertices.  Since $[u_1]\neq [u_2]$, there exists $z\in \ann(u_1) \backslash \ann(u_2)$.  But $[u_1]$ and $[u_2]$ have the same set of neighbors, hence $\ann(u_1) \backslash \{u_1\} = \ann(u_2) \backslash \{u_2\}$, forcing $z=u_1$ and hence $u_1^2=0$.  Let $[v]$ and $[w]$ be two adjacent vertices of $\Gamma_E(R)$ both of which are adjacent to $[u_1]$ and $[u_2]$. If $v^2=0$, then $[u_1+v]$ is adjacent to $[u_1]$ but not $[u_2]$, a contradiction.  A similar contradiction is reached if $w^2=0$.  If $v^2\neq 0$ and $w^2\neq 0$, then $[v+w]$ is not adjacent to either $[v]$ or $[w]$, a contradiction.  Thus $r=2$.

If $n_1>1$ and $n_2>1$, then there are two pairs of non-adjacent vertices, say $[u_1], [u_2]$ and $[v_1], [v_2]$.   Arguing as above, we have $u_1^2=0=v_1^2$.  But then $[u_1+v_1]$ is adjacent to both $[u_1]$ and $[v_1]$, a contradiction.  Thus either $n_1=1$ or $n_2=1$ and hence $\Gamma_E(R) = K_{n,1}$.
\end{proof}

A second look at the above results allows us to deduce some facts about cycle graphs, which are $n$-gons.  An immediate consequence of Proposition \ref{connected} is that there are no cycle graphs with eight or more vertices.  Likewise, Proposition \ref{bipartite} tells us that there are no 3- or 4-cycle graphs, i.e., $K_{1, 1, 1}$ or $K_{2,2}$, respectively.

\begin{proposition}  \label{nocycle} For any $R$, $\Gamma_E(R$) is not a cycle graph.
\end{proposition}

\begin{proof}  As mentioned above, it suffices to show that cycle graphs which contain a path of length five are not possible.  Suppose $[x]-[y]-[z]-[w]-[v]$ is a path in the graph.  Then $[yw]$ is annihilated by $[x]$, $[z]$, and $[v]$, but no such class exists.
\end{proof}

For a vertex $v$ of a simple graph $G$, the set of vertices adjacent to $v$ is called the \textit{neighborhood of $v$} and denoted $N_G(v)$, or simply $N(v)$ if the graph is understood.

\begin{lemma}\label{neighborhood_property}  Let $G$ be a finite, simple graph with the property that two distinct vertices $v$ and $w$ of $G$ are non-adjacent if and only if $N_G(v)=N_G(w)$.  Then $G$ is a complete $r$-partite graph for some positive integer $r$.
\end{lemma}

\begin{proof}  We induct on the number of vertices of $G$.  If $G$ has one vertex, then $G$ is a complete $1$-partite graph.  So suppose $G$ has more than one vertex.  Let $I=\{y_1, y_2, \ldots, y_n\}$ be a maximal set of pairwise non-adjacent vertices in $G$ (that is, $I$ is an independent set of $G$), and let $ N = \{x_1, \ldots, x_d\}$ be the common neighborhood of the elements of $I$.  Then $I$ and $N$ are disjoint and the vertices of $G$ are precisely $ I \cup  N$.  If $N=\varnothing$, then $G$ is $1$-partite and we are done.  So assume $N\neq \varnothing$.  Let $H$ be the graph induced by $ N$.  Note that $N_H(v) = N_G(v) \cap N$ and $N_G(v)=N_H(v) \cup  I$ for any vertex $v$ of $H$, hence two distinct vertices $v$ and $w$ of $H$ are non-adjacent if and only if $N_H(v)=N_H(w)$.  Since $H$ has fewer vertices than $G$, the induction hypothesis implies that $H$ is a complete $(r-1)$-partite graph for some $r\geq 2$.  But by the definition of $I$, this implies that $G$ is a complete $r$-partite graph.
\end{proof}

\begin{proposition}  \label{nonregular} For any ring $R$, there are no finite regular\footnote{A {\it regular graph} is one in which $\deg v =\deg w$ for every pair of vertices $v$ and $w$ of the graph.} graphs $\Gamma_E(R)$ with more than two vertices.
\end{proposition}

\begin{proof}  
Suppose that $R$ is a ring such that $\Gamma_E(R)$ is a regular graph of degree $d$ with at least $3$ vertices.  By earlier results, $d\geq 2$ and $|N([x])|=d$ for all $[x]\in \Gamma_E(R)$.  Since $\Gamma_E(R)$ is not complete, there exists two non-adjacent vertices $[y_1]$ and $[y_2]$.  If $N([y_1])\neq N([y_2])$, then without loss of generality we may assume that there is a vertex $[u] \in N([y_1]) \backslash N([y_2])$.  Thus $uy_1=0$ but $uy_2 \neq 0$.  This implies $[u] \in N([y_1y_2])$ and so $N([y_2])\neq N([y_1y_2])$.  But $N([y_2])\subseteq N([y_1y_2])$ and since each set has cardinality $d$ we must have equality, which leads to a contradiction.  Therefore, any two non-adjacent vertices on the graph have the same neighborhood, and clearly the converse is true.  Thus by Lemma \ref{neighborhood_property} and Proposition \ref{bipartite}, $\Gamma_E(R) =K_{n,1}$ for some $n\geq 1$.  If $\mathcal I=\{[y_1], [y_2], \dots, [y_n]\}$ is a maximal independent set in $\Gamma_E(R)$ and $\mathcal N = \{[x_1], \dots, [x_d]\}$ is the common neighborhood of the elements of $\mathcal I$, then $n\geq 2$ and by (the proof of) Lemma \ref{neighborhood_property}, $\mathcal I$ is one of the partitions of the vertex set of $\Gamma_E(R)$.  This implies that $d=1$, leading to a contradiction.  Thus $\Gamma_E(R)$ is not regular.
\end{proof}

Finally, we give the connection between $\Gamma_E(R)$ and the
zero-divisor graph $\Gamma(R)$ in \cite{AL}.  Let $\Gamma$ be a graph.  To each vertex $v_i$ of $\Gamma$, assign
an element $w_i \in \mathbb Z^{+} \cup \{\infty\}$, called the
weight of $v_i$, and let $w = (w_1, w_2, \ldots )$.  Define a
graph $\Gamma^{(w)}$ with vertex set $\{v_{k_i,i}\;|\; 1\leq k_i \leq w_i \}
$ and edge set is $\{ v_{k_i,i} v_{k_j,j} \;|\; i \neq j \mbox{ and } v_iv_j \mbox{
is an edge of }\Gamma\}.$
Intuitively, the vertices $v_{k_i,i}$ of $\Gamma^{(w)}$ form a
``covering set" of cardinality $w_i$ of the vertex $v_i$ of $\Gamma$,
and if $v_i$ is connected to $v_j$ in $\Gamma$, then every vertex of
the covering set of $v_i$ is connected to every vertex of the
covering set of $v_j$.

For each vertex $[v_i]$ of $\Gamma_E(R)$, let $w_i = |[v_i]|$, and let
$w=(w_1,w_2,\ldots )$. By the definitions of $\Gamma_E(R)$ and $\Gamma(R)$,
it is clear that $\Gamma_E(R)^{(w)}$ is a subgraph of $\Gamma(R)$.  In fact,
$\Gamma_E(R)^{(w)} = \Gamma(R)$ if and only if whenever $R$ contains a non-zero element $x$ such that $x^2=0$, then $|[x]| = 1$.  In particular, if $R$ is reduced, then $\Gamma_E(R)^{(w)} = \Gamma(R)$.  In general, $\Gamma(R)$ is the graph
obtained from $\Gamma_E(R)^{(w)}$ by adding edges between every pair of
vertices $v_{k_i,i}$ and $v_{k'_i,i}$ if and only if $v_i^2=0$.

\begin{example} The zero divisor graph $\Gamma(\mathbb Z_{12})$ is shown
``covering" the graph $\Gamma_E(\mathbb Z_{12})$.  Note that if $w = (2, 1, 2, 2)$, then $\Gamma_E(\mathbb Z_{12})^{(w)} = \Gamma(\mathbb Z_{12})$.

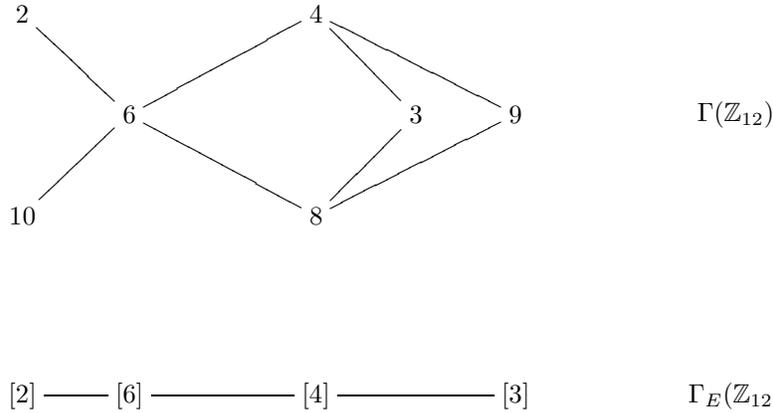
\xymatrix{
  & 2 \ar@{-}[dr]  & & & 4 \ar@{-}[dll]  \ar@{-}[dr] \ar@{-}[drr] \\
  & & 6 \ar@{-}[drr] \ar@{-}[dl] & & & 3 \ar@{-}[dl] & 9 \ar@{-}[dll] & & \Gamma(\mathbb Z_{12}) \\
  & 10 & & &  8 \\
  & \\
& [2] \ar@{-}[r] & [6] \ar@{-}[rr] & & [4] \ar@{-}[rr] & & [3] & & \Gamma_E(\mathbb Z_{12}) }
\end{example}

\section{Infinite Graphs and Fan Graphs}

In this section we investigate the properties of infinite graphs.  There are ostensibly two ways a graph can be infinite, namely it can have infinitely many vertices or it can have vertices with infinite degree.  More importantly, one might ask if the assumption that $R$ is Noetherian forces $\Gamma_E(R)$ to be a finite graph.  We begin with a discussion of degrees.

\begin{proposition}  \label{contain} Let $x$ and $y$ be elements of $R$.  If $0 \neq \ann(x) \subsetneq \ann(y)$, then $\deg[x] \leq \deg[y]$.
\end{proposition}

\begin{proof}  If $[u] \in \Gamma_E(R)$ such that $ux=0$, then clearly $uy=0$.
\end{proof}

\begin{proposition} \label{infvertexset} The vertex set of $\Gamma_E(R)$ is infinite if and only
if there is some associated prime $[x]$ maximal in $\mathfrak F$ such that $\deg[x]=\infty$.
\end{proposition}

\begin{proof} Clearly, if some vertex has infinite degree, then the vertex set of $\Gamma_E(R)$ is infinite.  Conversely, suppose the vertex set of $\Gamma_E(R)$ is infinite.  Let $[x_1], \dots, [x_r]$ be the maximal elements in $\frak F.$  If $\deg[x_1] < \infty$, then there are infinitely many vertices $[w]$ such that $wx_1 \neq 0$.  Now $[w][v] = 0$ for some class $[v] \neq [w]$, and since $[x_1]$ is prime, we must have $v \in \ann(x_1)$.  If there are infinitely many distinct vertices $[v]$, then $\deg[x_1] = \infty$, a contradiction.  Therefore, the set of $[v]$'s is finite, and hence $\deg[v] = \infty$ for some $v$.  Either $[v]$ is an associated prime of $R$ and maximal in $\frak F$, or $\ann(v) \subsetneq \ann(x_j)$ for some $j \neq 1$.  In either case, the result holds.
\end{proof}

\begin{definition}  A {\it fan} graph is a complete bipartite graph $K_{n, 1}$, $n \in \mathbb N\cup \{\infty\}$.  If $n=\infty$, we say the graph is an infinite fan graph.
\end{definition}

\begin{proposition}  \label{fan} Any ring $R$ such that $\Gamma_E(R)$ is a fan graph with at least four vertices satisfies the following properties:
\begin{enumerate}
\item Ass($R) = \{\mathfrak p \}$

\item $\mathfrak p^3 = 0$

\item char($R$) =  2, 4, or 8.
\end{enumerate}
\end{proposition}

\begin{proof}  Suppose we have a fan graph with at least four vertices, as show below.

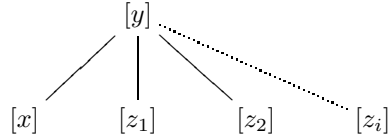
\xymatrix{
 & & & & & [y] \ar@{-}[dl] \ar@{-}[d] \ar@{-}[dr] \ar@{.}[drr] \\
 & & & & [x]  & [z_1]  & [z_2] & [z_i]}

Let $[y]$ be the unique vertex with maximal degree.  Note that $[y]$ is a maximal element in $\mathfrak F$, for if $\ann(y) \subsetneq \ann(w)$ for some $w$, then by Proposition \ref{contain} $\Gamma_E(R)$ would contain two vertices of degree larger than 1.  Next, in order for the classes $[x], [z_1], \dots$ to be distinct, we must have $z_i^2 = 0$, for all $i$.
If $i \neq j$, then $z_iz_j$ is non-zero and annihilated by $z_i$ and $z_j$; hence, the only choice for $[z_iz_j]$ is $[y]$.  Moreover, $(z_iz_j)^2 =0$ implies that every element of $[y]$ is nilpotent of order 2, by    \cite[(3.5)]{M}.   Since $[y]$ is connected to every vertex, it follows that  $z_iz_jz_k=z_iz_jx =0$ for any $i,j,k$.  As a result, we have $[xz_i] = [y]$.  Thus, $x^2z_i = 0$ for each $i$.  Consequently, either $x^2 = 0$ or $[x^2] = [y]$.  Since the latter case implies $x^3 = 0$, it follows that every zero divisor is nilpotent of order at most three.  This establishes (1) and (2).

For (3), consider $z_i + z_j$, which is annihilated by $z_iz_j$, but not $z_i$ or $z_j$.  Therefore, $[z_i+z_j]$ represents a degree one vertex distinct from $[z_i]$ and $[z_j]$.  The same conclusion can be reached for $[z_i-z_j]$.  Since $(z_i+z_j)(z_i-z_j) =0$, this means that $[z_i-z_j] =[z_i+z_j]$, and every element in the class is nilpotent of order two.  In particular, $(z_i + z_j)^2 = 0 \Rightarrow 2z_iz_j = 0$.  Consequently, either $2 \equiv 0$ or 2 is a zero divisor in $R$.  If char($R) \neq 2$, then $[2]$ is somewhere on the graph and either $2^2 = 0$ or $2^3 = 0$.
\end{proof}

\begin{corollary}  \label{finite_local} If $R$ is a finite ring and $\Gamma_E(R)$ is a fan graph with at least four vertices, then $R$ is a local ring.
\end{corollary}

\begin{proof}  A finite ring is a product of finite local rings, and the number of associated primes corresponds to the number of factors in the product; see, e.g., \cite{mcdonald}.
\end{proof}

It is unknown to the authors whether, for each positive integer $n$, the fan graph $K_{n, 1}$ can be realized as $\Gamma_E(R)$ for some ring $R$, or how one would go about the general construction or argument.  However, the following example shows that there exist rings $R$ such that $\Gamma_E(R)$ is a fan graph with infinitely many leaves.

\begin{example}  \label{inffan} Let $R = \mathbb Z_2[X,Y,Z]/(X^2, Y^2)$.  Then $\Gamma_E(R)$ is an infinite fan graph.  If $x$ and $y$ denote the images of $X$ and $Y$, respectively, then $\Ass(R) = \{(x,y)\}$, and the corresponding vertex $[xy]$ is the central vertex of the graph.  The leaves, besides $[x]$, $[y]$, are vertices of the form $[z^mx+z^ny]$ for each ordered pair $(m,n)$ of nonnegative integers.
\end{example}

\begin{remark}  Example \ref{inffan} shows that the Noetherian condition is not enough to force $\Gamma_E(R)$ to be finite.
\end{remark}

\section{Associated Primes}

One of the main motivations in studying graphs of equivalence classes of zero divisors is the fact that the associated primes of $R$, by their very definition, correspond to vertices in $\Gamma_E(R)$.  The focus here is the identification of associated primes of $R$, given $\Gamma_E(R)$.

It is important to note from Proposition \ref{contain} in the last section that proper containment of annihilator ideals does not translate into a strict inequality of degrees.  For instance, the ring $\mathbb Z[X,Y]/(X^3, XY)$, shown in Example \ref{introex} below, satisfies $\ann(y) \subsetneq \ann(x)$, but $\deg[y] = \deg[x]=2$.  In order to achieve strict inequality on degrees, further assumptions on the annihilator ideals are needed.  This issue is addressed in this section; but first we make some key observations.  The contrapositive of Proposition \ref{contain} is useful and worth stating:

\begin{proposition}  \label{>}   Let $x$ and $y$ be elements of $R$.  If $\deg[y] > \deg[x]$ for every $[x]$ in $\Gamma_E(R)$, then $\ann(y)$ is maximal in $\mathfrak F$ and hence is an associated prime.
\end{proposition}

Of course, we already saw evidence of this fact in Proposition \ref{fan}.   In that case, the vertex of maximal degree corresponds to the unique element of $\Ass(R)$.  Similarly, if $\Gamma_E(R)$ has exactly three vertices, then the vertex of degree two always corresponds to an associated prime maximal in $\mathfrak F$.  However, in this case $\Ass(R)$ may consist of more than one ideal.  For the example used in Corollary \ref{three}, the associated primes are $\ann(x) = (2, x, y)$ and $\ann(2) = (2, x)$.
This is just one instance where graphs with only two or three vertices prove the exception to the rule.

\begin{proposition}  \label{leaf} If $R$ is a ring such that $|\Gamma_E(R)| > 3$, then no associated prime of $R$ is a leaf.
\end{proposition}

\begin{proof}
Suppose $[y]$ is an associated prime of degree one.   Then there is one and only one $[x]$ such that $[x] \neq [y]$ and $[x][y] = 0$.  By Lemma \ref{assedge}, $\Ass(R) = \{\ann(x), \ann(y) \}$.  Any additional vertices, of which there are at least two, must be connected to $[x]$.  If $[z]$ is another vertex and $\deg [z] \geq 2$, then for some $[w]$, $zw = 0 \in \ann(y)$; hence $z \in \ann(y)$ or $w \in \ann(y)$.  Since this is not possible, it must be that $\deg [z] = \deg [w] =1.$  But in order for $[z]$ and $[w]$ to be distinct, we must have $z^2 = 0$ or $w^2 = 0$, which again puts $z \in \ann(y)$ or $w \in \ann(y)$.
\end{proof}

\begin{corollary}  \label{leafID} If $R$ is a ring such that $|\Gamma_E(R)| \geq 3$, then any vertex with a leaf is an associated prime maximal in $\mathfrak F$.
\end{corollary}

\begin{proof}  Proposition \ref{>} and Corollary \ref{three} take care of the case when there are exactly three vertices.  Suppose $|\Gamma_E(R)| > 3$.  Since every graph is connected, the result follows from Lemma \ref{assedge} and Proposition \ref{leaf}.
\end{proof}

We are now in a position to show a finer relationship between vertices of high degree and associated primes.

\begin{proposition}  \label{degrees} Let $x_1, \dots, x_r$ be elements of $R$, with $r \geq 2$, and suppose $\ann(x_1) \subsetneq \cdots \subsetneq \ann(x_r)$ is a chain in $\Ass(R)$.  If $3 \leq |\Gamma_E(R)| < \infty$, then $\deg[x_1]  < \cdots < \deg[x_r]$.
\end{proposition}

\begin{proof}
If $|\Gamma_E(R)| = 3$, then $|\Ass(R)| \leq 2$ by Lemma \ref{assedge} and Proposition \ref{incomplete}, and the result follows immediately from Corollaries \ref{three} and \ref{>}.  Thus, we may assume that $|\Gamma_E(R)| > 3$.  By Proposition \ref{fan}, $\Gamma_E(R)$ is not a fan graph, given the hypotheses.  By the preceding proposition, the degrees of $[x_1]$ and $[x_2]$ are at least two.
Suppose $\deg[x_1] = n.$  Then there are $n-1$ vertices $[u_i]$, distinct from either $[x_1]$ or $[x_2]$, such that $x_1u_i=0$ for all $i$.  Since $\ann(x_1) \subsetneq \ann(x_2)$, $x_2u_i = 0$ for all $i$ as well.  Furthermore, $x_2 \in \ann(x_1)$, hence $x_2^2 = 0$.   If $x_1^2 \neq 0$, then each $\ann(x_1 + u_i)$ contains $x_2$ but not $x_1$, making $\deg[x_2] > \deg[x_1]$.  If $x_1^2 = 0$, then take $z \in \ann(x_2) \backslash \ann(x_1)$.  Since $x_1 \notin \ann(z)$, $[z]$ is distinct from $[x_1]$ and $[x_2]$ and connected to $[x_2]$, but not $[x_1]$.  Again, $\deg[x_1] < \deg[x_2]$.  Moreover, this last argument applies to each pair $\ann(x_i)$, $\ann(x_{i+1})$, for $i \geq 1$, since $x_ix_j=0$ for $1 \leq i \leq j \leq r$.
\end{proof}

\begin{example} \label{introex} Let $R=\mathbb Z[X]/(X^3, XY)$ and let $x$ and $y$ denote the images of $X$ and $Y$, respectively.  Then $(x) \subsetneq (x,y)$ is a chain in $\Ass(R)$, corresponding to the vertices $[y]$ and $[x^2]$, respectively, with $\deg[y] = 2$ and $\deg[x^2] = 3$.

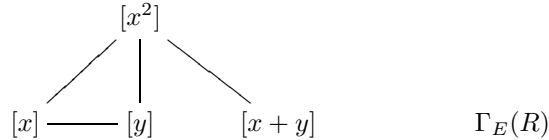
\xymatrix{
   & & & & & [x^2] \ar@{-}[dl] \ar@{-}[dr] \ar@{-}[d] \\
     & & & & [x] \ar@{-}[r] & [y] & [x+y] & & \Gamma_E(R)}
\end{example}

We now make use of Proposition \ref{nonregular} to establish a correspondence between associated primes and vertices of relatively large degree.

\begin{theorem}  \label{allmax} Let $R$ be such that $2 < |\Gamma_E(R)| < \infty$.  Then any vertex of maximal degree is maximal in $\mathfrak F$ and hence is an associated prime.
\end{theorem}

\begin{proof}  Let $d$ denote the maximal degree of $\Gamma_E(R)$.  If there is only one vertex of degree $d$, then Proposition \ref{>} yields the desired result.  Therefore, we may assume that $|\Gamma_E(R)| > 3$ and that $\Gamma_E(R)$ has at least two vertices of degree $d$.  Suppose that $[y_1]$ is a vertex of degree $d$.  Then $\ann(y_1) \subseteq \ann(y_2)$ for some $y_2$ such that $\ann(y_2)$ is maximal in $\frak F$, and so $[y_2]$ is an associated prime.  We aim to show that, in fact, $[y_1] = [y_2]$.  By Proposition \ref{contain}, $d=\deg[y_1]\leq \deg[y_2]\leq d$, so $\deg [y_2] = \deg [y_1]$.  Since $\ann(y_1) \subseteq \ann(y_2)$ yields $(N([y_1])\setminus \{[y_2]\})\subseteq (N([y_2])\setminus \{[y_1]\})$, the equality of degrees implies that $(N([y_1])\setminus \{[y_2]\})= (N([y_2])\setminus \{[y_1]\})$; denote this set by $\mathcal N$.  The connectivity of $\Gamma_E(R)$ implies that $d\geq 2$ and so $\mathcal N \neq \varnothing$.  

To get the desired result, we must show that $\ann(y_1) = \ann(y_2)$, that is, $\ann(y_2)\setminus \ann(y_1) = \varnothing$.  Suppose this is not the case and let $z \in \ann(y_2) \backslash \ann(y_1)$.  If $[z]$ is distinct from $[y_2]$ and $[y_1]$, then $[z]\in (N([y_2])\setminus \{[y_1]\})$ but $[z]\notin (N([y_1])\setminus \{[y_2]\})$, a contradiction.   Suppose $[z] = [y_2]$, in which case $y_2y_1 \neq 0$, $y_2^2 = 0$, and $y_1^2 \neq 0$.  Let $[w] \in \mathcal N$ and consider $w+y_2$, which is annihilated by $y_2$, but not $y_1$.  Then $[w+y_2] \notin \mathcal N$ and thus $[w+y_2]=[y_2]$.  So if $[v] \in \mathcal N$, then $0=(w+y_2)v=wv$.  Since $[w]$ was chosen arbitrarily, it follows that the vertices in $\mathcal N$ form a complete subgraph.  Hence $(\mathcal N \setminus \{[w]\}) \cup \{[y_2],[y_1]\} \subseteq N([w])$ for any $[w] \in \mathcal N$, implying that $\deg [w] \geq \deg [y_2]+1$, contradicting the maximality of $\deg [y_2]$.

Suppose $[z] = [y_1]$ and let $[x] \in \mathcal N$.  If $\deg [x] < d$, then there exists some $u \in \ann(y_1) \setminus \ann(x)$ such that $[u]$, $[y_1]$, and $[x]$ are distinct.  Note that $x(u+y_1)=xu \neq 0$, so $[u+y_1] \neq [y_i]$ for $i=1,2$.  However, $y_2(u+y_1) = 0$ which implies $[u+y_1] \in \mathcal N$, but $y_1(u+y_1) = y_1^2 \neq 0$, contradicting that $[u+y_1] \in \mathcal N$.  Therefore $\deg [x] = d$.  Furthermore, since $y_2(x+y_1) = 0$ but $y_1(x+y_1) = y_1^2 \neq 0$, then $[x+y_1] \in N([y_2])\setminus \mathcal N$, hence $[x+y_1] = [y_1]$.  Since $[x]$ was chosen arbitrarily in $\mathcal N$, it follows that $x_1x_2=0$ for distinct $[x_1],[x_2] \in \mathcal N$.  Therefore, the subgraph induced by $\mathcal N \cup \{[y_2],[y_1]\}$ is a complete graph.  By Proposition \ref{nonregular} there exists a vertex $[w]\in \Gamma_E(R)$ such that $\deg[w]<\deg[y_1]$, so $[w] \notin \mathcal N \cup \{[y_2],[y_1]\}$.  However, since $[y_2]$ is an associated prime, then by Lemma \ref{assedge}, $\mathrm{Ass}(R) \subseteq \mathcal N \cup \{[y_2],[y_1]\}$ and $[w]$ is adjacent to some vertex $[v] \in \mathcal N \cup \{[y_2],[y_1]\}$.  This implies that $\deg [v] > d$, a contradiction.
\end{proof}

\begin{example}  Recall that $\Gamma_E(\mathbb Z_4 \times \mathbb Z_4)$ in Example \ref{house} had two vertices, namely $[(2,0)]$ and $[(0,2)]$, of maximal degree.  Their annihilators are the associated primes of the ring.
\end{example}

\begin{remark}  We collect some comments regarding Theorem \ref{allmax}:
\begin{enumerate}
\item The converse is false:  For example, in $\mathbb Z_{p^2q^3}$, $[pq^3]$ and $[p^2q^2]$ correspond to associated primes $(p)$ and $(q)$, respectively, each maximal in $\mathfrak F$, but $\deg[pq^3] = 6$ and $\deg[p^2q^2] = 7$.

\item The assumption of finiteness is necessary.  The graph of the ring $R = \mathbb Z[X,Y]/(X^2, Y^2, Z^2)$ has four vertices with infinite degree, namely $[xy], [xz],$ $[yz]$, and $[xyz]$, but only $\ann(xyz)$ is an associated prime of $R$.
\end{enumerate}
\end{remark}

Finally, since each of our examples possesses a leaf, one might ask if Corollary \ref{leafID}  supersedes  Theorem \ref{allmax}; that is, if every graph must contain a leaf.  Based on the results in section one, the first case where such an example can occur is when the graph has exactly five vertices.

\begin{example}  Let $R = \mathbb Z_3[X,Y]/(XY, X^3, Y^3, X^2-Y^2)$, and let $x$ and $y$ denote the images of $X$ and $Y$ in $R$, respectively.

\xymatrix{
  & & & [x] \ar@{-}[dr] \ar@{-}[rr] & & [y] \ar@{-}[dl] \\
     & & & & [x^2] \ar@{-}[dr] \ar@{-}[dl] & & & \Gamma_E(R) \\
     & & & [x+y] \ar@{-}[rr] & & [x+2y]}
\end{example}

\end{document}